\documentclass[a4paper,12pt]{amsart}
\usepackage[T1]{fontenc}
\usepackage[utf8]{inputenc}
\usepackage[english]{babel}
\usepackage{amsmath,amssymb,amsthm,mathabx,verbatim}
\usepackage[all]{xy}
\usepackage[bookmarks=true]{hyperref}
\usepackage{color}
\usepackage{graphicx}
\usepackage{cancel}
\usepackage{datetime}
\usepackage{booktabs}
\usepackage[text={170mm,230mm},centering]{geometry}

\newtheorem{thm}{Theorem}[section] %teorema numerato
\newtheorem{prop}[thm]{Proposition} %proposizione numerata
\newtheoremstyle{named}{}{}{\itshape}{}{\bfseries}{.}{.5em}{#3}
\theoremstyle{named} 
\theoremstyle{remark}
\newtheorem{rem}[thm]{Remark} %remark numerato
\theoremstyle{definition}
\newenvironment{sistema}%
{\left\lbrace\begin{array}{@{}l@{}}}%
{\end{array}\right.}

\newcommand{\ZZ}{\ensuremath{\mathbb{Z}}}
\newcommand{\QQ}{\ensuremath{\mathbb{Q}}}
\newcommand{\RR}{\ensuremath{\mathbb{R}}}
\newcommand{\CC}{\ensuremath{\mathbb{C}}}
\newcommand{\FF}{\ensuremath{\mathbb{F}}}
\newcommand{\PP}{\ensuremath{\mathbb{P}}}
\newcommand{\HH}{\ensuremath{\mathcal{H}}}

\newcommand{\cfs}{\ensuremath{\mathcal{S}}}

\newcommand{\SL}{\ensuremath{\mathrm{SL}}}
\newcommand{\GL}{\ensuremath{\mathrm{GL}}}

\newcommand{\Gammait}{\ensuremath{\mathit{\Gamma}}}

\renewcommand{\Im}{\mathop{\mathrm{Im}}}

\newcommand{\Xns}{X_{\textnormal{ns}}}
\newcommand{\Xs}{X_{\textnormal s}}
\newcommand{\Cs}{C_{\textnormal s}}
\newcommand{\Gammaits}{\Gammait_{\textnormal s}}
\newcommand{\eqdef}{\displaystyle\mathop{=}^\textnormal{def}}

\begin{document}

\title{Double Covers of Cartan Modular Curves}
\author{Valerio Dose}
\email{dose@mat.uniroma2.it}

\thanks{Il primo autore ha svolto questo lavoro come titolare di un Assegno ``Ing. Giorgio Schirillo'' dell'Istituto Nazionale di Alta Matematica}

\author{Pietro Mercuri}
\email{mercuri.ptr@gmail.com}

\address{Dipartimento di Matematica\\Universit\`a di Roma ``Tor Vergata''\\
Via della Ricerca Scientifica 1\\00133 Roma\\ITALY}

\author{Claudio Stirpe}
\email{clast@inwind.it}

\subjclass[2010]{14G35,11G18,14Q05}
\keywords{modular curves, double cover, elliptic curves, galois representations}

\begin{abstract}
We present a strategy to obtain explicit equations for the modular double covers associated respectively to both a split and a non-split Cartan subgroup of $\text{GL}_2(\mathbb F_{p})$ with $p$ prime. Then we apply it successfully to the level $13$ case.
\begin{comment}
We give explicit equations for the modular curves $\Xs(13)$ and $\Xns(13)$ associated respectively to a split and a non-split Cartan subgroup of $\text{GL}_2(\mathbb F_{13})$. We provide them as double covers of the isomorphic modular curves $\Xs^+(13)$ and $\Xns^+(13)$ associated to the normalizers of such subgroups. We compute two models with different properties for each curve and we give explicit maps between them.
\end{comment}
\end{abstract}

\maketitle

\section{Introduction}

Non-cuspidal rational points on modular curves parametrize elliptic curves over $\mathbb Q$ with particular properties regarding the associated Galois representation modulo some positive integer~$N$. Having equations for such modular curves helps to explicitly determine elliptic curves with a given Galois representation modulo $N$. When $N=p$ is a prime number, modular curves associated to maximal subgroups of $\text{GL}_2(\mathbb F_p)$ with surjective determinant, such as normalizers of Cartan subgroups, play an important role in the framework of Serre's uniformity problem, which concerns the determination of elliptic curves with surjective Galois representation modulo~$p$.

Recent work has been done in determining equations of modular curves of Cartan type. An  affine plane model for $X_0(169)\cong \Xs(13)$ is computed in \cite{Kenku169}\cite{Kenku169errata}. An equation for $\Xns^+(13)$ and $\Xs^+(13)$ is computed in \cite{Baran13}, while models of $\Xns^+(p)$ for $p=17, 19$ can be found in \cite{Merns}. An equation for $\Xns(11)$ and the double cover $\Xns(11)\rightarrow\Xns^+(11)$ is computed in \cite{DFGS}. Furthermore, it is used in \cite{Fermat23n} to solve the generalized Fermat equation with exponents $2,3,11$ and in \cite{ZywinaPossible} to classify the possible images of Galois representations modulo $11$ associated to elliptic curves. 

In this paper we present a strategy to compute birational models of modular curves $\Xs(p)$ and $\Xns(p)$ obtained as double covers of their explicitly given quotients $\Xs^+(p)$ and $\Xns^+(p)$, associated respectively to the normalizer of a split and a non-split Cartan subgroup of $\textnormal{GL}_2(\mathbb F_p)$. As an application, we compute singular models in $\mathbb A^3$ for the genus 8 curves $\Xs(13)$ and $\Xns(13)$. For completeness and further check we also compute smooth equations of the canonical model of such curves in $\mathbb P^7$ using classical methods and give explicit maps between all the known models of $\Xs(13)$ and $\Xns(13)$.

In particular, in Section \ref{sec:canmod} we recall the method to compute equations of the canonical model, in Section \ref{doublecovers} we describe the new strategy in general,  and in Section \ref{Results} we apply the two methods to the level 13 case, obtaining the equations. Furthermore, in Section \ref{maps} we also compute a birational map between the two models found.\newline

\begin{comment}
In this paper we compute models for the genus $8$ modular curves $\Xs(13)$ and $\Xns(13)$ in two different ways. We find equations in $\mathbb P^7$ for the canonical model of these curves starting from modular forms and following the methods of \cite{Galbraith}, \cite{Mer0} and \cite{Merns}.
Then, we also compute a singular affine model in $\mathbb A^3$ by finding functions on $\Xs^+(13)$ and $\Xns^+(13)$ whose square roots generate the function fields of $\Xs(13)$ and $\Xns(13)$. In particular, in Sections \ref{sec:canmod} and \ref{doublecovers} we describe the two methods in general, and in Section \ref{Results} we apply the two methods to the level 13 case, obtaining the equations. Furthermore, in Section \ref{maps} we also compute a birational map between the two models found.\newline 
\end{comment}

\noindent \textbf{Acknowledgments.} We would like to thank Prof. R. Schoof for an early review and useful comments on this paper. We would also like to thank Prof. M. Caboara for giving us access to computational resources at the University of Pisa, and M. Paganin for his bibliographic help.

\section{Notation and basic facts}

Let $\HH=\{\tau\in\CC, \Im(\tau)>0\}$ be the complex upper half-plane and let $\HH^*=\HH\cup\{\infty\}\cup\QQ$, both endowed with the action of $\SL_2(\mathbb R)$ given by fractional linear transformation. Given a congruence subgroup $\Gammait$ of $\text{SL}_2(\mathbb Z)$, we can consider the modular curve $X(\Gammait)$ associated to $\Gammait$ which is obtained by providing the orbit space $\Gammait\backslash\mathcal H^*$ with the structure of a compact Riemann surface. When $\Gammait=\SL_2(\ZZ)$ we denote the associated modular curve of genus 0 by $X(1)$. The complex points of $X(1)$ parametrize elliptic curves over $\mathbb C$ up to isomorphism. There is a morphism from every modular curve $X(\Gammait)$ to the modular curve $X(1)$ which is given by the group inclusion $\Gammait\subset\text{SL}_2(\mathbb Z)$. This morphism is called the $j$-map of $X(\Gammait)$.

Let $N$ be a positive integer and let $H$ be a subgroup of $\GL_2(\ZZ/N\ZZ)$. We can associate to $H$ the congruence subgroup $\Gammait_H\eqdef \{x \in \SL_2(\ZZ)\textnormal{ such that }x \pmod N \in H\}$ and the modular curve $X_H\eqdef X(\Gammait_H),$ which admits the structure of projective algebraic curve. When the determinant homomorphism $\textnormal{det}: H\rightarrow (\ZZ/N\ZZ)^{\times}$ is surjective, $X_H$ can be defined over $\mathbb Q$. Furthermore, if the matrix $\begin{pmatrix} -1 & 0\\0 & -1\end{pmatrix}$ belongs to $H$, then, for any number field $K$, the $K$-rational points on $X_H$ parametrize the elliptic curves defined over $K$ such that $H$ contains the image of the associated Galois representation modulo $N$, given by the action of the Galois group $\textnormal{Gal}(\overline{\mathbb Q} /K)$ on $N$-torsion points.

If we take $H$ to be a Borel subgroup of $\GL_2(\ZZ/N\ZZ)$, we obtain the classical modular curve $X_0(N)$. This curve has an automorphism $w_N$, the Atkin-Lehner involution, induced by the action of the matrix $\begin{pmatrix}0 & -\frac1{\sqrt N} \\ \sqrt N & 0\end{pmatrix}\in\SL_2(\RR)$ on $\mathcal H$. This automorphism allows us to define the quotient curve $X_0^+(N)\eqdef X_0(N)/\langle w_N\rangle$.

Let now $N=p$ be a prime number. If we take $H$ to be a split or non-split Cartan subgroup of $\GL_2(\ZZ/p\ZZ)$ (see \cite[p. 278, Section 2.1]{Ser}) we denote the associated modular curve by $\Xs(p)$ and $\Xns(p)$, respectively. Any Cartan subgroup of $\GL_2(\ZZ/p\ZZ)$ has index $2$ in its normalizer. Thus we obtain involutions $w_\textnormal{s}$ and $w_\textnormal{ns}$ and double covers $\Xs(p)\rightarrow \Xs^+(p)$ and $\Xns(p)\rightarrow\Xns^+(p)$, where $\Xs^+(p)$ and $\Xns^+(p)$ are the modular curves associated to the normalizer of a split and a non-split Cartan subgroup of $\GL_2(\ZZ/p\ZZ)$, respectively. Moreover we have $\Xs^+(p)=\Xs(p)/\langle w_\textnormal{s}\rangle$ and $\Xns^+(p)=\Xns(p)/\langle w_\textnormal{ns}\rangle$. Since the congruence subgroup of $\SL_2(\ZZ)$ associated to a split Cartan subgroup of $\GL_2(\ZZ/p\ZZ)$ is conjugate to the congruence subgroup associated to a Borel subgroup of $\GL_2(\ZZ/p^2\ZZ)$, we get the isomorphism $\Xs(p)\cong X_0(p^2)$. Furthermore, the involution $w_s$ of $\Xs(p)$ corresponds to the Atkin-Lehner involution of $X_0(p^2)$, so that also $\Xs^+(p)$ is isomorphic to $X_0^+(p^2)$.

\section{Explicit equations of modular curves using the canonical embedding} \label{sec:canmod}

In this section we briefly recall how to get explicit equations for the canonical model of modular curves when we know enough Fourier coefficients of a basis of the vector space of modular forms corresponding to the space of differentials of the curve (see \cite{Mer0} for more details).

 Let $\Gammait$ be a congruence subgroup of $\SL_2(\ZZ)$, let $\cfs_2(\Gammait)$ be the $\CC$-vector space of the cusp forms of weight $2$ with respect to $\Gammait$.  We know that $\cfs_2(\Gammait)$ is isomorphic to the $\CC$-vector space of holomorphic differentials $\Omega^1(X(\Gammait))$ via the map $f(\tau)\mapsto f(\tau)d\tau$ (see \cite[p. 81, Theorem 3.3.1]{DS}). Using this isomorphism when the genus $g$ of $X(\Gammait)$ is greater than~$2$, we get the following realization of the canonical map
\begin{align*}
\varphi\colon X(\Gammait) & \to \PP^{g-1}(\CC) \\
\Gammait\tau & \mapsto (f_1(\tau):\ldots:f_g(\tau)),
\end{align*}
where $\tau\in\HH^*$ and $\mathcal B=\{f_1,\ldots,f_g\}$ is a $\CC$-basis of $\cfs_2(\Gammait)$. The Enriques-Petri Theorem (see \cite[Chapter 4, Section 3, p. 535]{GH} or \cite{Sd73}), states that the canonical model of a complete non-singular non-hyperelliptic curve is entirely cut out by quadrics and cubics. 

%Moreover, it is cut out only by quadrics if it is neither trigonal, nor a quintic plane curve of genus $6$.

Though the Enriques-Petri Theorem is proved only over algebraically closed fields, when $X(\Gammait)$ can be defined over $\QQ$, we can try to look for quadratic and cubic equations over $\QQ$ for the image of $\varphi$. Then we can check if the zero locus $Z$ of such equations, which contains by construction the image of $\varphi$, is an algebraic curve with the same genus as $X(\Gammait)$. If this is the case, an application of the Riemann-Hurwitz formula tells us that the morphism $\varphi\colon X(\Gammait)\to Z$ is an isomorphism.

We know that $X_0(N)$ is not hyperelliptic when $N>71$ (\cite[p. 451, Theorem 2]{OggHyp}) while $X_0^+(p^r)$, with $p$ prime and $r$ a positive integer, is not hyperelliptic when its genus is bigger than $2$ (\cite[p. 370, Theorem B]{Has97}). Furthermore, $\Xns(p)$ is not hyperelliptic when $p\geq 11$ and $\Xns^+(p)$ is not hyperelliptic when $p\geq 13$ (\cite[p. 76, Theorem 1.1]{DoseCartan}). Hence, in all these cases, we need to look only for equations of degree $d=2\text{, }3$ for the image of $\varphi$. This is done in the following way.
\subsection{Algorithm description}\label{alg:bettermodel}
Consider the basis $\mathcal{B}=\{f_1,\dots,f_g\}$ of $\mathcal S_2(\Gammait)$. We can think of the $f_i$ as power series in $\mathbb C[[q]]$ through their Fourier $q$-expansion. Let's fix the degree $d$ of the equations and suppose we know the first $m$ coefficients of such power series, with $m>d(2g-2)$. This condition on $m$ guarantees that if we have a polynomial $F$ with rational coefficients and $g$ unknowns such that $F(f_1,\ldots,f_g)\equiv 0 \pmod{q^{m+1}}$, then $F(f_1,\ldots,f_g)=0$ (see \cite[Section 2.1, Lemma 2.2, p. 1329]{BGGP05}). Now we evaluate all the monomials of degree $d$ in the basis $\mathcal{B}$, obtaining elements of $\mathbb C[[q]]$ of which we know the first $m$ coefficients. In this way we get $m$ vectors generating a subspace $S$ of $\mathbb C^k$ where $k$ is the number of monomials of degree $d$ with fixed coefficient. 
A basis of the space $S^{\perp}$, the orthogonal space to $S$ in $\mathbb C^k$, gives the coefficients of the desired equations for the image of $\varphi$. These coefficients belong to a number field that depends on the choice of the basis $\mathcal{B}$ and it is not necessarily $\QQ$. Using representation theory of $\GL_2(\ZZ/p\ZZ)$, one can find a basis $\mathcal{B}$ of cusp forms such that the equations obtained are defined over $\mathbb Q$ (see \cite{Merns} for more details).

%A basis of the space $S^{\perp}$, the orthogonal space to $S$ in $\mathbb C^k$, gives the coefficients of the desired equations for the image of $\varphi$. These coefficients belong to a number field that depends on the choice of the basis $\mathcal{B}$ and it is not necessarily $\QQ$. Using representation theory of $\GL_2(\ZZ/p\ZZ)$ one can find, starting from a basis of normalized eigenforms and using linear algebraic transformations, a basis $\mathcal{B}$ of cusp forms such that each element of $\mathcal{B}$ is invariant with respect to the action of the Galois group $\textnormal{Gal}(\overline{\mathbb Q} /\mathbb Q)$. We call this basis a $\mathbb Q$-basis, meaning that it is defined over $\mathbb Q$. Notice that the Fourier coefficients of a $\mathbb Q$-basis of cusp forms are not necessarily in $\mathbb Q$, but the basis of the space $S^{\perp}$ in $\mathbb C^k$ obtained using a $\mathbb Q$-basis is actually in $\mathbb Q^k$. This gives equation defined over $\mathbb Q$.

Once we have equations defined over $\QQ$ we can assume that their coefficients belong to $\ZZ$. In this case we are interested in reducing the size of the coefficients of these equations and minimizing the number of primes $\ell$ such that the model has bad reduction modulo~$\ell$. There is a number field $K$ which contains all the coefficients of all the elements of $\mathcal{B}$ (\cite[p. 234, Theorem 6.5.1]{DS}). We can assume that the Fourier coefficients of the basis $\mathcal{B}$ are algebraic integers, hence their coordinates with respect to a suitable chosen basis of $K$ over $\QQ$ belong to~$\ZZ$. To reduce the size of the coefficients of the equations we can apply the LLL algorithm, first to the Fourier coefficients of $\mathcal{B}$ and then to the $\ZZ$-basis of the space $S^{\perp}$. We know that if the elements of $\mathcal{B}$ are linearly dependent modulo a prime number $\ell$, then the canonical model of the curve that we find is singular modulo $\ell$. In \cite[Algorithm 2.1]{Mer0} there is a description of how to modify $\mathcal{B}$ such that the elements of $\mathcal{B}$ are linearly independent for each prime~$\ell$.

%There is a number field $K$, of degree~$D$ over~$\QQ$, which contains all the coefficients of all the elements of $\mathcal{B}$. Let $M$ be the matrix of the coordinates with respect to a suitable chosen basis of $K$ over $\QQ$. The matrix $M$ has $g$ rows and $mD$ columns, where $m$ is the number of the first Fourier coefficients used. Since $m>d(2g-2)$ and $g>0$, we always have $m\geq g$, hence the rank of $M$ is $g$. We assume that the Fourier coefficients of the basis $\mathcal{B}$ are algebraic integers, hence the entries of $M$ belong to~$\ZZ$.

%Now, we describe how to find the primes $\ell$ such that the canonical model has bad reduction. Let $\mathcal{I}_{\text{jac}}$ be the ideal generated by the polynomials defining the curve and by all the determinants of order $g-2$ of the jacobian matrix of the curve. We compute the elimination ideals $\mathcal{J}_i:=\mathcal{I}_{\text{jac}}\cap \QQ[x_i]$, for $i=1,\ldots,g$. If $\mathcal{J}_i\neq 0$, it turns out to be generated by $\lambda x_i^n$ with $\lambda,n\in\ZZ_{>0}$. The curve has bad reduction modulo any prime~$\ell$ such that $\ell\mid \lambda$.

\section{Equations of modular double covers}\label{doublecovers}

Let $\pi\colon X\rightarrow Y$ be a double cover of modular curves of the type $\Xs(p)\rightarrow\Xs^+(p)$ or $\Xns(p)\rightarrow \Xns^+(p)$ where $p$ is a prime number. We have $j_X=j_Y\circ \pi$ where $j_X$ and $j_Y$ are the $j$-maps to $X(1)\cong\mathbb P^1$ of respectively $X$ and $Y$. In this section we describe a general strategy for determining equations for $X$ and $\pi$ starting from existing equations of $Y$ and $j_Y$. This is  basically the same strategy used in \cite{DFGS} to find equations for the modular double cover $X_{\text{ns}}(11)\rightarrow X_{\text{ns}}^+(11)$ and we describe here how it can be applied in general. In section~\ref{Results} we apply it successfully to the double covers $\Xs(13)\rightarrow \Xs^+(13)$ and $\Xns(13)\rightarrow \Xns^+(13)$.

Let $K$ be the function field over $\mathbb Q$ of $Y$ and let $p_1=0,\dots,p_k=0\in\mathbb Q[x_1,\dots,x_h]$ be affine equations for $Y$. We have $K\cong \mathbb Q(x_1,\dots,x_h)/(p_1,\dots,p_k)$. Let now $L$ be the function field of $X$ over $\mathbb Q$, and let $K\subset L$ be the field inclusion given by the morphism $\pi$. We would like to express $L$ in the form
$$L=K(\sqrt{q})$$
where $q$ is a square-free polynomial in $K$, so that equations for $X$, possibly singular, are $p_1=0,\dots,p_k=0,t^2-q=0$ in the variables $x_1,\dots,x_h,t$. We start by observing the following facts.
\begin{prop}\label{ramification}
Let $L=K(\sqrt{q})$ for some non-square element $q$ in $K$. Then the points of $Y$ over which the morphism $\pi$ ramifies are exactly the zeros and poles of odd order for~$q$.
\end{prop}
\begin{proof}
This is a consequence of Kummer Theory. (\cite[p. 122, Proposition 3.7.3(b)]{Stichtenoth})
\end{proof}
\begin{prop}\label{qfunction}
Let $q$ be a function in $K$ such that $L=K(\sqrt q)$ and let $f$ be another function of $K$ whose zeros and poles of odd order are the points of $Y$ over which the morphism $\pi$ ramifies. If the rational 2-torsion of the Jacobian of $Y$ is trivial, then $L=K(\sqrt{\lambda f})$ for some constant $\lambda \in \mathbb Q$.
\end{prop}
\begin{proof}
Proposition \ref{ramification} implies that all the zeros and poles of odd order for both $q$ and $f$ are the points of $Y$ over which the morphism $\pi$ ramifies. This means that the function $q/f$ has zeros and poles of even order, i.e. $\text{div}(q/f)=2D$ for some degree 0 divisor $D$ of $Y$. Since $q$ and $f$ have rational coefficients, the triviality of the rational 2-torsion in the Jacobian implies that $\text{div}(q/f)=\text{div}(h^2)$ for some $h\in K$. Thus $q/f=\lambda h^2$ for some constant $\lambda\in\mathbb{Q}$, so that $\sqrt q$ and $\sqrt{\lambda f}$ generate the same field over $K$.
\end{proof}
Therefore, assuming the hypothesis of Proposition \ref{qfunction}, if we find a function $f$ on $Y$ whose zeros and poles of odd order are the points over which $\pi$ ramifies, then we have determined the function field of $X$ to be of the form $L=K(\sqrt{\lambda f})$ for some constant $\lambda\in\mathbb Q$. Note that in this situation, $f$ and $\lambda$ can be both multiplied by a square (of respectively a rational function or a rational constant), and the field $L$ would still be generated by $\sqrt{\lambda f}$ over $K$. We can then suppose $f$ to be a polynomial function.

\subsection{Determining the ramification points}

The points over which the morphism $\pi$ ramifies could be determined if we have equations for the modular $j$-map $j_Y:Y\rightarrow X(1)\cong\mathbb P^1$. In fact $\pi$ can ramify only over points in which the function $j_Y$ has a pole, a zero or it is equal to $1728$ (see \cite[Chapter 2, Section 2.3]{DS}). In other words, $\pi$ can ramify over a cusp or over those elliptic points of $Y$ whose preimage by $\pi$ does not contain elliptic points of $X$. Furthermore, in both our cases, the number of cusps in $X$ is exactly the double of the number of cusps in $Y$ (see \cite[p.454, Proposition 3]{OggHyp}, \cite[Proposition 7.10]{BaranClass}). 

Let $p$ be a prime and $r$ a positive integer. A study of the elliptic points of $\Xns(p^r)$ and $\Xns^+(p^r)$ can be found in \cite{BaranClass}, while the analogous study for $X_0(N)$ is classical for any positive integer $N$ (\cite[p. 92, Section 3.7]{DS}) and therefore for $\Xs(p^r)\cong X_0(p^{2r})$.
In \cite[Section 2]{OggHyp} the author computes the number of ramification points of the map $X_0(N)\rightarrow X_0^+(N)$ (and therefore of $\Xs(p^r)\rightarrow \Xs^+(p^r)$), but he does not determine the order of the elliptic points associated. %The complete study of the elliptic points of $\Xs(p^r)$ and $\Xs^+(p^r)$ is contained in section \ref{EllipticSplit}. 

\begin{rem}\label{EllipticPointsRemark}
A direct computation with cosets representatives shows that the number of elliptic points $e_2$ and $e_3$ for the curve $\Xs(p^r)$ is
\[
e_2=\begin{cases}
2 & \text{if } p\equiv 1 \pmod 4 \\
0 & \text{otherwise}% p\equiv 3 \pmod 4 \text{ or } p=2,
\end{cases}\quad
e_3=\begin{cases}
2 & \text{if } p\equiv 1 \pmod 3 \\
0 & \text{otherwise}% p\equiv 2 \pmod 3 \text{ or } p=3.
\end{cases}
\]
while for the curve $\Xs^+(p^r)$ we have
\[
e_2=\begin{cases}
1+\frac{p^{r-1}}{2}(p-1) & \text{if } p\equiv 1 \pmod 4 \\
\frac{p^{r-1}}{2}(p+1) & \text{if } p\equiv 3 \pmod 4 \\
2^{r-1} & \text{if } p=2
\end{cases}\qquad
e_3=\begin{cases}
1 & \text{if } p\equiv 1 \pmod 3 \\
0 & \text{otherwise}% p\equiv 2 \pmod 3 \text{ or } p=3.
\end{cases}
\]
\end{rem}

\subsection{Determining \texorpdfstring{$f$}{f}}\label{Searchingf}
Once we have determined the coordinates of the points $P_1,\dots,P_r$ of $Y$ over which $\pi$ ramifies, one can try to find a rational function $f$ on $Y$, whose zeros and poles of odd order are exactly $P_1,\dots,P_r$. A procedure for doing this, which does not necessarily gives a solution, but it has been successful for $\Xns^+(11)$, $\Xns^+(13)$ and $\Xs^+(13)$ is the following.

Let $D$ be the divisor on $Y$ given by $P_1+\dots+P_r$. This divisor is defined over $\mathbb Q$ because $\pi$ is defined over $\mathbb Q$. Furthermore $r$ must be even because of the Riemann-Hurwitz formula. Let $Q_1,\dots,Q_k$ be the expected rational points on $Y$, which are only the CM-points of class number one, when $Y=\Xns^+(p)$ or the CM-points of class number one plus the rational cusp, when $Y=\Xs^+(p)$. Then we compute Riemann-Roch spaces of the type
$$B(n_1,\dots,n_k)\mathop{=}^\textnormal{def}H^0\left(-D+\sum_{i=1}^k n_iQ_i\right)$$
where $n_i$ is even for every $i$ and $\displaystyle\sum_{i=1}^kn_i=r$. Since the divisor $\displaystyle -D+\sum_{i=1}^k n_iQ_i$ has degree $0$ then $B(n_1,\dots,n_k)$ has dimension $0$ or $1$. If we find such a space of dimension $1$ and we can give a basis of rational functions of it, then we have found our function $f$.

\subsection{Verifying the triviality of the rational 2-torsion in the Jacobian}

Let $A$ be an abelian variety over $\mathbb Q$ and let $\ell\neq 2$ be a prime of good reduction for $A$. If $A$ has a nontrivial rational point of $2$-torsion, then the number of points of $A$ over $\mathbb F_{\ell}$ must be even because $\ell\neq 2$ and the reduction modulo $\ell$ is an isomorphism on the $2$-torsion points of $A$. 
 Hence, in our case, to prove the triviality of the rational $2$-torsion in the Jacobian $\text{Jac}(Y)$ of $Y$, it is enough to find a prime number $\ell\neq 2,p$ such that the quantity $\#\textnormal{Jac}(Y)(\mathbb F_{\ell})$ is odd.

Computing $\#\textnormal{Jac}(Y)(\mathbb F_{\ell})$ can be done recalling that the Jacobian of $Y$ is isogenous to some factor of the Jacobian of $X_0^+(p^2)$, the whole Jacobian in the split case, and the new part in the non-split case. Therefore we need to compute the number of $\mathbb F_{\ell}$-rational points of such factor which can be done using the Eichler-Shimura relations that relate the characteristic polynomial of the $\ell$-th Frobenius endomorphism acting on the Jacobian of $X_0(p^2)$ with the characteristic polynomial of the Hecke operator $T_{\ell}$ (for an example of this, see the proof of \cite[p. 76, Theorem 1.1]{DoseCartan}.

\subsection{Determining \texorpdfstring{$\lambda$}{lambda}}\label{lambda}

Once we have equations of $X$ up to multiplication by a constant $\lambda\in Q$, we can determine $\lambda$ by analyzing the field of definitions of special values of the function $f$.

Let $Q\in Y$ be a rational CM-point, therefore of class number one. Then the two points in $X$ over $Q$ are defined over the CM-field of the elliptic curve associated to $Q$ (see \cite[pp. 194-195]{SerreMordell}). This allows us to determine the constant $\lambda$ because $\sqrt{\lambda f(Q)}$ must generate the CM-field.

\section{Results for level 13}\label{Results}

We recall that $\Xs^+(13)$ and $\Xns^+(13)$ are both curves of genus 3. They are isomorphic, and an equation for both of them is the following quartic in $\mathbb P^2$ 
\begin{align}\label{BurcuEq}
p(X,Y,Z)\eqdef&(-Y-Z)X^3+(2Y^2+ZY)X^2+(-Y^3+ZY^2-2Z^2Y+Z^3)X+\\+&(2Z^2Y^2-3Z^3Y)=0 \nonumber
\end{align}
Starting from this equation it is possible to obtain formulas for the different $j$-maps of both $\Xs^+(13)$ and $\Xns^+(13)$ (see \cite{Baran13} for details and explicit formulas) which we will call respectively $j_\textnormal{s}$ and $j_\textnormal{ns}$.

\subsection{Singular Equations of \texorpdfstring{$\Xs(13)$ and $\Xns(13)$}{Xs(13) and Xns(13)} in \texorpdfstring{$\mathbb{A}^3$}{A3}}\label{singularmodels}

Here we apply the strategy described in Section \ref{doublecovers} to obtain equations of $\Xs(13)$ and $\Xns(13)$ starting from the known same equation \ref{BurcuEq} of $\Xs^+(13)$ and $\Xns^+(13)$.\newline

We begin by looking for the coordinates (in the model given by equation \ref{BurcuEq}) of the points of $\Xs^+(13)$ and $\Xns^+(13)$ over which the two modular double covers $\pi_\textnormal{s}\colon \Xs(13)\rightarrow \Xs^+(13)$ and $\pi_\textnormal{ns}\colon\Xns(13)\rightarrow\Xns^+(13)$ ramify. Remark \ref{EllipticPointsRemark} and \cite[p. 2768, Proposition 7.10]{BaranClass} tell us that both these double covers ramify over six elliptic points of order $2$. These implies that such points are among the simple zeros of the function $j-1728$. In \cite[Appendix A]{Baran13}, we can find explicit affine formulas for $j_\textnormal{s}$ and $j_\textnormal{ns}$ which are both in the form $j(x,y)=h(x,y)/k(x,y)$, where $h$ and $k$ are polynomials with integer coefficients, $x\eqdef X/Z$ and $y\eqdef Y/Z$. 

Then to find the simple zeros of $j(x,y)-1728$ we compute the resultant with respect to $x$ of the polynomial $p(x,y,1)$, defining the affine equation of the curve, and the polynomial $h(x,y)-1728 k(x,y)$. This resultant has the form:
$$(x-1)(43x^6 - 194x^5 - 115x^4 + 692x^3 + 85x^2 - 498x + 243)\varphi(x) $$
for the split case and
$$(2888x^6 + 12500x^5 + 13443x^4 + 24786x^3 + 134781x^2 + 230254x + 120131)\psi(x)$$
for the non-split case, where $\varphi$ and $\psi$ are a product of irreducible polynomials with multiplicity higher than one. Using MAGMA (\cite{MAGMA}) we can check that both the polynomials of degree~6, made explicit above, define functions on $\Xs^+(13)\cong\Xns^+(13)$ with 18 simple zeros, which are divided in two Galois orbits of cardinality 12 and 6. We call $P_1,\dots,P_6$ the simple zeros in the Galois orbit of cardinality 6. Since the modular double covers $\pi_\textnormal{s}$ and $\pi_\textnormal{ns}$  are defined over $\mathbb Q$, the set of points that ramify in one of them must be composed of entire Galois orbits. This implies that such points are exactly the points $P_1,\dots,P_6$.\newline

The Jacobian of $\Xs^+(13)\cong\Xns^+(13)$ doesn't have any rational torsion (\cite[p.60, Example 12.9.3.]{BPSgenus3}). Hence, Propositions \ref{ramification} and \ref{qfunction} imply that we can determine the function fields of $\Xs(13)$ and $\Xns(13)$ up to a constant, by adding to the function field of $\Xs^+(13)\cong\Xns^+(13)$ the square root of a function $f$ whose zeros and poles of odd order are the points $P_1,\dots,P_6$. Following the procedure described in Section \ref{Searchingf}, and using MAGMA, we compute a basis of each Riemann-Roch space of type $$H^0\left(-(P_1+\dots+P_6)+\sum_{i=1}^7 n_iQ_i\right)$$ 
with $n_1,\dots,n_7$ even integers such that $-12\leq n_1,\dots,n_7\leq 12$ and $\displaystyle \sum_{i=1}^7 n_i=6$, and where $Q_1,\dots,Q_7$ are the 7 rational points of $\Xs^+(13)\cong\Xns^+(13)$ associated to the rational CM-points of class number one, or to the rational cusp in the case of $\Xs^+(13)$ (see \cite[p. 275, Table 1.1]{Baran13}).

With these conditions on the coefficients, the divisors $\displaystyle -(P_1+\dots+P_6)+\sum_{i=1}^7 n_iQ_i$ have degree~$0$, hence the associated Riemann-Roch spaces have dimension 0 or 1. There are indeed a few of these spaces which are non-trivial. Among them, we choose one that is associated to a divisor of the lower degree for which the associated space is non-trivial. For the split case we choose the space generated by 
\begin{align*}
f_\textnormal{s}(x,y)=&\frac{6y^5 + 18y^4 - 17y^3 - 37y^2 - 5y + 3}{y^4}x^2 + \\
+&\frac{-5y^5 - y^4 + 18y^3 - 17y^2 - y + 6}{y^3}x + \\
+&\frac{10y^5 - 6y^4 - 38y^3 + 28y^2 + 14y - 3}{y^4},
\end{align*}
which is a function of degree 10, and we choose for the non-split case the space generated by
\begin{align*}
f_\textnormal{ns}(x,y)=&\frac{-11y^8 - 20y^7 - 41y^6 - 85y^5 - 260y^4 - 586y^3 - 635y^2 - 312y - 56}{y^4}x^2 +\\
  +& \frac{22y^9 + 29y^8 + 51y^7 + 109y^6 + 372y^5 + 866y^4 + 1124y^3 + 840y^2 + 340y +
    56}{y^4}x + \\
    +& \frac{-11y^9 + 2y^8 - y^7 + 8y^6 - 70y^5 - 68y^4 - 453y^3 - 1016y^2 - 740y
    - 168}{y^3},
\end{align*}
which is a function of degree 12. We can get rid of denominators by multiplying $f_\textnormal{s}$ and $f_\textnormal{ns}$ by $y^4$ which is a square in the function field of $\Xs^+(13)\cong\Xns^+(13)$. We obtain the polynomial functions 
\begin{align*}
q_\textnormal{s}(x,y)=&(6y^5 + 18y^4 - 17y^3 - 37y^2 - 5y + 3)x^2 + 
(-5y^6 - y^5 + 18y^4 - 17y^3 - y^2 + 6y)x + \\
+&10y^5 - 6y^4 - 38y^3 + 28y^2 + 14y - 3\\
&\qquad\\
q_\textnormal{ns}(x,y)=&(-11y^8 - 20y^7 - 41y^6 - 85y^5 - 260y^4 - 586y^3 - 635y^2 - 312y - 56)x^2 +\\
  +&(22y^9 + 29y^8 + 51y^7 + 109y^6 + 372y^5 + 866y^4 + 1124y^3 + 840y^2 + 340y +
    56)x + \\
    +& -11y^{10} + 2y^9 - y^8 + 8y^7 - 70y^6 - 68y^5 - 453y^4 - 1016y^3 - 740y^2
    - 168y.
\end{align*}
Now we can determine the constant $\lambda$ of Section \ref{lambda}. Note that 
\begin{align*}
q_\textnormal{s}(0,0)=-3 &\qquad q_\textnormal{s}(0,3/2)=-3^3\cdot 2^{-4}\\ q_\textnormal{ns}(-1,0)=-7\cdot 2^4 &\qquad q_\textnormal{ns}(0,3/2)=-163\cdot 2^{-10}\cdot 3^{10}
\end{align*}
which is consistent with \cite[p. 275, Table 1.1]{Baran13}. We have thus obtained singular models in $\mathbb A^3$ for $\Xs(13)$ and $\Xns(13)$ given respectively by the equations $\begin{sistema}p(x,y,1)=0 \\ t^2-q_\textnormal{s}(x,y)=0\end{sistema}$ and $\begin{sistema}p(x,y,1)=0 \\ t^2-q_\textnormal{ns}(x,y)=0\end{sistema}$. For both curves, the double cover over $\Xs^+(13)\cong\Xns^+(13)$ has equation $(x,y,t)\mapsto (x,y)$.

\subsection{Smooth Equations of \texorpdfstring{$\Xs(13)$ and $\Xns(13)$}{Xs(13) and Xns(13)} in \texorpdfstring{$\mathbb{P}^7$}{P7}}\label{canonicalmodels}

To find equations describing $\Xs(13)$ it is enough to take a basis for $\mathcal{S}_2(\Gammait_0(169))$ using the software MAGMA (\cite{MAGMA}) or William Stein's tables (\cite{SteinTab}). Then operate as shown in the proof of Corollary 6.5.6 in \cite[p. 238]{DS} to find a basis with rational integer Fourier coefficients. Finally, we apply the method explained in Section \ref{sec:canmod}. The equations obtained are shown in the Appendix.

The equations for the map $\pi_{\textnormal s}\colon \Xs(13)\to \Xs^+(13)$, using this model for $\Xs(13)$ are obtained in the following way. Let $\mathcal{B}=\{f_1,\ldots,f_g\}$ be a basis of eigenforms for $\mathcal{S}_2(\Gammait_0(p^2))$ and we assume that the first $g^+$ elements of $\mathcal{B}$ are invariant with respect of the action of $w_{p^2}$. Hence, $\mathcal{B}^+=\{f_1,\ldots,f_{g^+}\}$ is a basis for $\Omega^1(\Xs^+(p))$. So, the canonical embedding give a canonical model $\mathcal{C}$ in $\PP^{g-1}$ for $\Xs(p)$ using $\mathcal{B}$ and a canonical model $\mathcal{C}^+$ in $\PP^{g^+-1}$ for $\Xs^+(p)$ using $\mathcal{B}^+$, and we have the morphism
\begin{align*}
\pi\colon \mathcal{C} &\longrightarrow \mathcal{C}^+, \\
(x_1:\ldots:x_{g^+}:x_{g^++1}:\ldots:x_g)&\longmapsto(x_1:\ldots:x_{g^+}).
\end{align*}

Let $t$ and $t^+$ be the compositions of the invertible projective linear transformations of $\mathbb P^{g-1}$ and $\mathbb P^{g^+-1}$ respectively, that we use in algorithm \ref{alg:bettermodel} to obtain better models of $\Xs(p)$ and $\Xs^+(p)$.
\begin{comment}
Let $t$ and $t^+$ be the compositions of the invertible linear transformations of $\Omega^1(\Xs(p))$ and $\Omega^1(\Xs^+(p))$ respectively, that we use in algorithm \ref{alg:bettermodel} to obtain better models of $\Xs(p)$ and $\Xs^+(p)$. Let $\mathcal{B}_1$ and $\mathcal{B}^+_1$ be the basis obtained by applying $t$ and $t^+$ to $\mathcal{B}$ and $\mathcal{B}^+$.

We get the models $\mathcal{C}_1$ for $\Xs(p)$ using $\mathcal{B}_1$ and $\mathcal{C}^+_1$ for $\Xs^+(p)$ using $\mathcal{B}^+_1$. 
\end{comment}
 We have the following commutative diagram

\[
\begin{xy}
(0,20)*+{\mathcal{C}}="z"; (20,20)*+{\mathcal{C}_1}="x";
(0,0)*+{\mathcal{C}^+}="y"; (20,0)*+{\mathcal{C}^+_1}="yy";
{\ar@{->} "z";"x"}?*!/_2mm/{t};
{\ar "z";"y"}?*!/^2mm/{\pi};
{\ar "x";"yy"}?*!/_3mm/{\pi_1};
{\ar@{->} "y";"yy"}?<<<<<<*!/^3mm/{t^+};
\end{xy}
\]
where $\pi_1$ is just the composition $t^{-1}\circ\pi\circ t^+$. In the case $p=13$, $\mathcal C_1^+$ is the model given by equation \ref{BurcuEq} and $\pi_1$ is $\pi_\textnormal{s}$.
\bigskip

Finding equations describing $\Xns(13)$ is more difficult because we don't have the Fourier coefficients of a basis for $\mathcal{S}_2(\Gammait_{\text{ns}}(13))$. To find these Fourier coefficients we use a basis of $\mathcal{S}_2(\Gammait_0(169))^{\text{new}}$ and some representation theory of $G:=\GL_2(\ZZ/p\ZZ)$. We use the newforms because the jacobian of $X_{ns}(p)$ is isogenous over $\mathbb Q$ to the new part of the jacobian of $X_0(p^2)$ ( see \cite{Chen}, \cite{Edix}).

The irreducible complex representations of the finite group $G$ are divided into three kinds: representations of dimension $p-1$, representations of dimension $p$ and representations of dimension $p+1$. The representations of dimension $p-1$ are also called cuspidal representations (the name has nothing to do with cusp forms) and this kind of representations is parametrized by characters $\theta\colon\FF_{p^2}^*\to \CC^*$. The other two kinds of representations are called principal series representations and are parametrized by characters $\mu$ of the upper triangular matrices subgroup of $G$. We have a representation of dimension $p$ if $\mu$ is the quadratic character, and we have a representation of dimension $p+1$ otherwise.

Let $V_f$ be the $\CC[G]$-span of an element $f$ of a basis of eigenforms for $\mathcal{S}_2(\Gammait_0(p^2))^{\text{new}}$. We know that $V_f$ is a complex irreducible representation that is a principal series representation if $f$ is a twist of a form $h$ of $\mathcal{S}_2(\Gammait_1(p))$,
i.e. if the Fourier coefficient $a_n$ of $f$ is equal, for each $n$, to $\chi(n)b_n$, where $\chi$ is a character of $\FF_p$ and $b_n$ is the $n$-th Fourier coefficient of $h$. The dimension of $V_f$ is $p$ if the form $h$ is in $\mathcal{S}_2(\Gammait_0(p))$ and is $p+1$ otherwise. If $f$ is not a twist of a low level form, then $V_f$ is a cuspidal representation. Then, one can find elements invariant under the action of a non-split Cartan subgroup of $G$ using the related trace.

We have that $\mathcal{S}_2(\Gammait_0(169))^{\text{new}}$ has dimension $8$ and let $\mathcal{B}=\{f_1,\ldots,f_8\}$ be a basis of eigenforms of $\mathcal{S}_2(\Gammait_0(169))^{\text{new}}$. Three of the forms in $\mathcal{B}$ are conjugate with respect to the Galois action and form a basis for the $w_N$-invariant forms in $\mathcal{S}_2(\Gammait_0(169))^\textnormal{new}$;
they are not twist of some lower level form, so the irreducible representations associated are all cuspidal. Two Galois conjugate forms in $\mathcal{B}$ are twists of a form of $\mathcal{S}_2(\Gammait_1(13))^{\text{new}}$, which is a complex vector space of dimension $2$. The last three forms in $\mathcal{B}$ are conjugate with respect to the Galois action and they are not twists of some lower level form, so the associated irreducible representations are all cuspidal.
The equations obtained are written in the Appendix together with the equations for the map $\pi_{\textnormal{ns}}\colon \Xns(13)\to \Xns^+(13)$ which are obtained analogously to the split case.

\section{Maps from the canonical models to other models}\label{maps}

To compute maps from the canonical model $\mathcal{C}$ to a different model $\mathcal{C}'$, we use the reverse-mapping correspondence between curves and function fields. What we do is finding an injective field homomorphism $\iota$ from the function field $\mathcal F'$ of $\mathcal{C}'$ to the function field $\mathcal F$ of the canonical model $\mathcal C$. To achieve this, we need a way to go from rational functions on the canonical model to their Laurent $q$-expansion and vice versa. One direction is easy. Indeed, we know that the $x_i$ in the equations of the canonical models in Section \ref{canonicalmodels} correspond to elements in a specific basis of cusp forms, that we found beforehand. On the other hand recognizing Laurent $q$-series as rational functions in the $x_i$ requires more work.

%We know that the function field of the canonical model is generated by ratios of differentials. We also know that the differentials have a Fourier $q$-expansion obtained via the isomorphism between the differentials space and the weight $2$ cusp forms space. Hence, we can express each element of $\mathcal{F}$ as a Laurent $q$-series. 

Lets place ourself in the affine chart of $\mathcal C$ where $x_8\neq 0$. The function field of $\mathcal C$ is generated by the functions $\frac{x_1}{x_8},\ldots,\frac{x_7}{x_8}$ which are all well defined in the affine chart we chose. 
Let $h_i$ be the Laurent $q$-expansion of $\frac{x_i}{x_8}$ for $i=1,\ldots,7$. Let $f$ be an element of $\mathcal F'$ and suppose we know the Laurent $q$-expansion of $\iota(f)\in \mathcal F$. We want to write $\iota(f)$ in the form
\[
\iota(f)=\frac{p(h_1,\ldots,h_7)}{q(h_1,\ldots,h_7)},
\]
where $p$ and $q$ are suitable polynomials. We write the previous equality as
\[
p(h_1,\ldots,h_7)-\iota(f)q(h_1,\ldots,h_7)=0.
\]
where the left hand side above can be seen as a linear combination of Laurent $q$-series, assuming we know the degree of the polynomials $p$ and $q$. 
 
%We can express each monomial of the polynomials $p(h_1,\ldots,h_n)$ and $q(h_1,\ldots,h_n)$ as Laurent $q$-series. Since we know the Laurent $q$-expansion of $\iota(f)$, we can express as Laurent $q$-series also the elements $\iota(f)M(h_1,\ldots,h_n)$, where $M(h_1,\ldots,h_n)$ is any monomial of $q(h_1,\ldots,h_n)$. 

Therefore, if we know the first $m$ Laurent coefficients of $\iota(f),h_1,\dots,h_7$, with $m>d(2g-2)=14d$ and $d$ is the maximum of the degrees of $p$ and $q$, it is easy to compute the coefficients of $p$ and $q$ in the same way explained in the algorithm \ref{alg:bettermodel} description for the coefficients of the polynomial $F$, i.e. we have $m$ vectors generating a subspace $S$ and we want a basis of $S^{\perp}$. If we don't know the degree of $p$ and $q$, we make computations just trying sufficiently large degree of $p$ and $q$ until we find some non-trivial relations among Laurent coefficients.

\subsection{Map to Kenku's affine plane model of $X_0(169)$}
\allowdisplaybreaks
In \cite{Kenku169}\cite{Kenku169errata}, Kenku gives an explicit plane affine model of $X_0(169)$ which is naturally isomorphic to $\Xs(13)$. Let $X$ and $Y$ be the coordinates in the affine model of $X_0(169)$ described in Kenku's paper. They correspond to Puiseux $q$-series obtained by
\[
X(\tau)=\frac{13\eta^2(169\tau)}{\eta^2(\tau)}, \qquad Y(\tau)=\frac{\eta^2(\tau)}{\eta^2(13\tau)},
\]

where $\eta$ is the classical Dedekind eta function. Here we have Puiseux series instead of Laurent series but the method is the same. We consider the affine model for $\Xs(13)$ as the affine chart where $x_8\neq 0$ in the projective model described by equations in Section \ref{canonicalmodels}. Then we have the field isomorphism
\begin{align*}
\iota\colon \mathbb Q(X_0(169))&\longrightarrow  \mathbb Q(\Xs(13))\\
X&\longmapsto U\\
Y&\longmapsto V,
\end{align*}
where
\begin{align*}
\mathrm{numerator}(U)&=117x_1^2 - 13x_1x_2 + 13x_1x_3 + 13x_4x_6 + 13x_4x_7 + \\
&+ 26x_4x_8 - 13x_5^2 +    13x_6x_8 + 13x_7^2, \\
\mathrm{denominator}(U)&=238x_3^2 + 215x_3x_4 + 215x_3x_5 + 429x_3x_6 - 419x_3x_7 + 36x_3x_8 - 89x_4^2 +   \\
&+185x_4x_5 + 130x_4x_6 - 505x_4x_7 - 313x_4x_8 + 305x_5^2 + 217x_5x_6 +   \\
&+145x_5x_7 - 46x_5x_8 + 28x_6^2 - 7x_6x_7 + 352x_6x_8 + 351x_7^2 - 3x_7x_8 -  2x_8^2; \\
\mathrm{numerator}(V)&=4637022x_1^2 + 4624659x_4x_6 - 5060016x_4x_7 + 14784393x_4x_8 - 6782997x_5^2+ \\
& -    19275477x_5x_6 - 8559018x_5x_7 + 1545960x_5x_8 - 28694289x_6^2 -    8134854x_6x_7 +\\
&- 4473261x_6x_8 + 6858072x_7^2 + 2366208x_7x_8 - 3989778x_8^2, \\
\mathrm{denominator}(V)&=-209376188x_3^2 - 196485388x_3x_4 - 183091120x_3x_5 - 421799299x_3x_6 +    \\
&+371436944x_3x_7 - 136573881x_3x_8 + 89731271x_4^2 - 151182225x_4x_5 +\\
&-    140218639x_4x_6+ 488527387x_4x_7 + 280939604x_4x_8 - 277852129x_5^2 +\\
&-    207933217x_5x_6 - 146929317x_5x_7 + 17764144x_5x_8 - 15033364x_6^2 + \\
&+   3885141x_6x_7 - 323708963x_6x_8 - 329322311x_7^2 - 3989778x_7x_8.
\end{align*}

\subsection{Desingularization maps to the affine models of section \ref{singularmodels}}\label{desingularization}

The function fields defined by the equations found in Section \ref{canonicalmodels} and the function fields defined by the equations found in Section \ref{singularmodels} are both isomorphic to the function field of the associated modular curve, which is $\Xs(13)$ or $\Xns(13)$. Here we give an explicit isomorphism between the function fields defined by the two models, in both the split and the non-split case.

Let $\mathcal C$ be the smooth projective model defined in Section \ref{canonicalmodels} and let $\mathcal C'$ be the singular affine model defined in Section \ref{singularmodels}. We have the following situation
$$
\xymatrix
{
\mathcal C \ar[dr]_\pi\ar@{-->}[rrr]^\varphi & & &   \mathcal C'\ar[dl]^{\pi'}\\
& \mathcal C^+_\textnormal{p}& \mathcal C^+_\textnormal{a}\ar@{_{(}->}[l]
}
$$
where $\mathcal C^+_\textnormal{p}$ is the curve defined by equation \ref{BurcuEq}, $\mathcal C^+_\textnormal{a}$ is the affine chart in which $Z\neq 0$, the map $\pi$ is $\pi_\textnormal{s}$ or $\pi_\textnormal{ns}$ depending on whether we are dealing with the split or the non-split case, $\pi'$ is the double cover given by $(x,y,t)\mapsto (x,y)$ and $\varphi$ is a birational map that makes the diagram commute on some affine chart of $\mathcal C$. The isomorphism from the function field of $\mathcal C'$ generated by $x,y,t$ to the function field of $\mathcal C$ is given in the form
\begin{align*}
\varphi^* \colon \mathbb Q(\mathcal C')&\mathop{\longrightarrow}^\cong \mathbb Q(\mathcal C)\\
x&\longmapsto X/Z\\
y&\longmapsto Y/Z\\
t&\longmapsto \tilde t\\
\end{align*}
where $X,Y,Z$ are the one defined in Section \ref{canonicalmodels} in the equations of $\pi_\textnormal{s}$ and $\pi_\textnormal{ns}$, and $\tilde t\eqdef s\cdot(Y/Z)^2$, with $s$ being a square root of $\pi_\textnormal{s}^* f_\textnormal{s}$ or $\pi_\textnormal{ns}^* f_\textnormal{ns}$ depending on whether we are dealing with the split or the non-split case. To determine $s$ we take a square root of the Laurent $q$-expansion of $\pi_\textnormal{s}^* f_\textnormal{s}$ or $\pi_\textnormal{ns}^* f_\textnormal{ns}$ and then we recognize it as a rational function in the $x_i$, as explained in the beginning of Section \ref{maps}. In the split case we get
\begin{equation*}
s=\frac{4x_1 - x_2 - x_3 + x_4 - 3x_5 - x_6 - 2x_7 + x_8}{-x_2 + x_3 + x_4}
\end{equation*}
and for the non-split case we get
\begin{align*}
\mathrm{numerator}(s)&=78953974807x_1^2 + 26x_1x_2 - 25x_1x_3 - x_1x_4 + 2x_3^2 +\\ 
&+238115162692x_4x_7 +     209337250703x_4x_8 - 582346348536x_5^2 +\\
&+ 727177285412x_5x_6 + 78542213920x_5x_7 - 563548816331x_5x_8 +\\ 
&+ 65380280758x_6^2 - 244488381626x_6x_7 + 82647686352x_6x_8 +\\
&+ 136959277010x_7^2 +     250609762421x_7x_8 - 257891423548x_8^2, \\
\mathrm{denominator}(s)&=33279035581x_3^2 - 20440236060x_3x_4 + 161001990516x_3x_5 +\\
&+ 177481085270x_3x_6     - 284601313488x_3x_7 - 214125958084x_3x_8+\\
 &- 116902000189x_4^2 +     103367036819x_4x_5 + 124067876928x_4x_6+\\
  &- 155405328616x_4x_7 -     193679032128x_4x_8 - 57688123584x_5^2+\\
   &- 123976858837x_5x_6 -     194732784800x_5x_7 - 165341806053x_5x_8+\\
    &- 126114432327x_6^2 +     524882113804x_6x_7 + 271440452599x_6x_8+\\
     &- 236487356215x_7^2 -     365606104840x_7x_8 - 113208254802x_8^2.
\end{align*}
\allowdisplaybreaks[0]

\section{Appendix: Equations for the canonical models of $\Xs(13)$ and $\Xns(13)$}

\allowdisplaybreaks

The curve $\Xs(13)$ of $g=8$ can be explicitly given by the following 15 equations in $\PP^{7}$.

\begin{align*}
&x_1x_2 - x_1x_3 - x_2^2 - x_2x_4 + x_2x_5 + x_3x_6 - x_3x_7 - x_4x_6 - x_4x_7 - x_4x_8=0 \\
&-x_1^2 + 2x_1x_2 + x_1x_4 - x_1x_5 + x_1x_6 + x_3^2 + x_3x_5 + x_3x_6 - x_3x_7 - x_4^2 + x_4x_5 - x_4x_7 + \\
&\quad - x_5x_8 + x_6^2 + x_6x_8 + x_7^2=0 \\
&x_1^2 + x_1x_3 - x_1x_4 + x_1x_5 + x_1x_7 + x_1x_8 - x_3x_4 + x_3x_7 + x_3x_8 + x_4^2 - 2x_4x_5 - x_5x_6 + \\
&\quad +x_5x_7 + x_5x_8 - x_6^2=0 \\
&-x_1x_6 - 2x_1x_8 + x_2^2 + x_2x_4 - x_2x_5 - x_3x_6 - x_3x_8 - x_4x_6 + x_4x_8 + x_5x_6 - x_5x_7 + \\
&\quad - x_5x_8 + x_6^2=0 \\
&-x_1x_2 + x_1x_3 - x_1x_4 - 2x_1x_6 + x_2^2 - x_3^2 - x_3x_5 - 2x_4^2 + x_4x_6 +x_5^2 + x_5x_6=0 \\
&x_1x_2 - x_1x_4 + x_2^2 + 2x_2x_4 - x_2x_5 + x_3x_6 + x_3x_7 - x_3x_8 + x_4x_6 + x_5x_6 - x_5x_7=0 \\
&-x_1x_2 + x_1x_6 - x_2^2 - x_2x_5 - x_3x_4 + x_3x_5 - x_3x_6 + x_3x_8 + x_4x_6 + x_4x_8 - x_5x_6 + x_5x_7 + \\
&\quad +x_5x_8 - x_6^2 + x_6x_8 + x_7^2=0 \\
&x_1x_2 - x_1x_3 + x_2x_3 + x_2x_4 + x_2x_5 + x_3x_7 - x_3x_8 + x_4^2 - x_4x_6 + x_4x_8 - x_5^2 - x_5x_7 + \\
&\quad - x_6x_8 - x_7^2=0 \\
&-x_1x_3 - x_1x_4 - x_1x_6 - 2x_2x_4 + x_2x_6 + x_3x_8 - x_4x_7 - x_5x_6=0 \\
&x_1x_2 - x_1x_8 - x_2x_5 - x_2x_7 + x_2x_8 - x_3x_7 - x_4x_5 - x_4x_7 + x_5^2 - x_5x_8 - x_6x_7 + x_6x_8 + \\
&\quad +x_7^2=0 \\
&2x_1x_3 + x_1x_4 + x_1x_7 - x_1x_8 + x_2^2 + x_2x_5 + x_2x_7 + x_3^2 + x_3x_6 - x_3x_7 + x_3x_8 - x_4x_5 + \\
&\quad - x_4x_6 - x_4x_7 + x_5x_6 + x_5x_7 + x_6^2 + x_6x_8 + x_7^2=0 \\
&x_1x_3 - x_1x_4 - x_1x_5 - x_1x_6 - x_2x_4 + x_2x_6 - x_2x_8 - 2x_3x_4 - x_3x_6 - x_3x_7 + x_3x_8 - x_4^2 + \\
&\quad - x_4x_5 + x_6x_7=0 \\
&x_1x_5 - x_1x_6 - x_1x_7 + x_2x_4 + x_2x_5 - x_2x_6 - x_2x_8 + x_3x_4 - x_3x_8 - x_4x_6 - x_4x_7 - x_4x_8 + \\
&\quad +x_5^2 + x_5x_6 + x_6x_7 - x_6x_8 - x_7^2=0 \\
&-x_1^2 - x_1x_2 - x_1x_4 + x_1x_6 + x_2^2 + x_2x_4 - x_2x_8 + x_3x_4 - 2x_3x_5 + x_3x_7 + x_4^2 + x_4x_6 + \\
&\quad - x_4x_8 - x_5^2 + x_6x_7 - x_6x_8 - x_7^2=0 \\
&x_1^2 - x_1x_2 - x_1x_3 + x_1x_7 + 2x_1x_8 + x_2x_3 - x_2x_7 + x_3^2 + x_3x_4 + x_3x_5 + x_4^2 + 2x_4x_5 + \\
&\quad - x_5x_6 - x_5x_8=0.
\end{align*}
\allowdisplaybreaks[0]

This curve has only two rational points: the two cusps (\cite[p. 241, Theorem 1]{Kenku169},\cite{Kenku169errata}). Using the previous equations, these rational points have the following coordinates.
\[
\begin{array}{|c|}
\toprule
\text{Rational points} \\
\midrule
(-2:-1:-4:3:6:-3:1:4) \\
(0:0:0:0:0:0:0:1) \\
\bottomrule
\end{array}
\]

The map $\pi_{\textnormal s}\colon \Xs(13)\to \Xs^+(13)$, using the previous model for $\Xs(13)$ and the model (\ref{BurcuEq}) for $\Xs^+(13)$, is
\[
\begin{sistema}
X= - x_1 + x_2 + 2x_4 + x_5 - x_6 + x_7 - x_8 \\
Y= - x_2 - x_3 + x_4 + x_5 - x_6 - x_8 \\
Z= - x_1 - x_2 - 2x_4 + x_5 + x_6.
\end{sistema}
\]
\bigskip

The curve $\Xns(13)$ of $g=8$ can be explicitly given by the following 15 equations in $\PP^{7}$.

\allowdisplaybreaks
\begin{align*}
&x_1^2 - x_1x_3 - x_1x_4 - x_1x_7 + x_1x_8 + x_2x_4 + x_2x_5 + 2x_3x_4 - 2x_3x_5 - x_3x_8 + 2x_4x_5 + \\
&\quad+x_4x_7 + x_5x_8 - x_7^2 + x_7x_8=0, \\
&-x_1x_3 + 2x_1x_5 + x_1x_8 - 2x_3x_4 - x_3x_5 + x_3x_6 - x_3x_7 - x_4x_5 - x_4x_6 + x_4x_7 + \\
&\quad +x_4x_8 - x_5^2 + x_5x_6 - 3x_5x_8 - x_6x_7 - 3x_6x_8 + x_7^2 - x_8^2=0, \\
&-x_1x_3 + 2x_1x_4 + x_1x_5 - 2x_1x_6 + 4x_1x_8 + x_2x_4 + x_2x_5 - x_3x_4 + x_3x_6 - x_3x_7 - x_4^2 + \\
&\quad +x_4x_5 - 2x_4x_8 + 2x_5x_7 + x_5x_8 - 2x_6x_8 + x_7x_8 - x_8^2=0, \\
&x_1x_3 + x_1x_4 + x_1x_5 - 3x_1x_6 + x_1x_7 + 2x_1x_8 - x_2x_3 - x_2x_4 + x_2x_5 + x_2x_6 - x_3^2 + \\
&\quad - x_3x_4 - x_3x_5 + x_3x_6 - x_3x_8 - 2x_4x_5 - x_4x_8 + x_5x_6 + x_5x_7 + 2x_6^2 - 2x_6x_7 + x_7^2 + \\
&\quad +x_7x_8 - x_8^2=0, \\
&x_1x_2 - x_1x_3 + x_1x_5 + x_1x_6 - x_1x_7 + x_1x_8 + x_2^2 + x_2x_3 - x_2x_4 - x_2x_5 - x_2x_6 + x_3^2 + \\
&\quad - x_3x_4 - x_3x_5 - x_3x_6 + x_3x_8 - x_4^2 + x_4x_5 + 2x_4x_6 + x_4x_7 - 2x_4x_8 - x_5^2 + 2x_5x_6 + \\
&\quad +x_5x_7 - 2x_5x_8 + x_6x_7 - x_6x_8 + x_7x_8 - x_8^2=0, \\
&2x_1x_2 + x_1x_3 - x_1x_4 + x_1x_6 - x_1x_7 - x_1x_8 + x_2x_3 - 2x_2x_4 - x_2x_5 - x_2x_6 + x_2x_7 + \\
&\quad +x_3^2 - 2x_3x_4 - x_3x_6 - x_3x_7 + x_4x_5 + x_4x_6 + x_4x_7 + 2x_4x_8 + x_5x_6 - 2x_5x_8 + \\
&\quad +x_6x_7 - x_6x_8=0, \\
&-x_1^2 + x_1x_2 + 2x_1x_3 + x_1x_5 - x_1x_6 - x_1x_7 + 2x_1x_8 - x_2^2 - x_2x_3 + x_2x_6 + x_2x_7 + \\
&\quad +x_2x_8 - x_3x_4 - x_3x_5 - x_3x_8 - x_4^2 + x_4x_6 + x_5^2 - x_5x_6 - x_6x_7 - x_6x_8=0, \\
&-x_1^2 - x_1x_2 + x_1x_5 + 2x_1x_6 + x_1x_7 + x_1x_8 + x_2x_3 - x_2x_4 - x_2x_5 + x_2x_7 + x_2x_8 - x_3x_5 + \\
&\quad +x_3x_6 - x_3x_7 + x_4x_5 - x_4x_6 + x_4x_7 - x_5^2 + x_5x_6 + x_5x_7 - x_5x_8 - 2x_6x_8 + x_7x_8 - x_8^2=0, \\
&-2x_1x_2 + 2x_1x_3 - x_1x_4 - x_1x_5 + x_1x_7 - x_1x_8 - x_2x_4 + 2x_2x_5 + 2x_2x_6 + x_2x_8 - x_3^2 + \\
&\quad +x_3x_4 + x_3x_5 + x_3x_6 + x_3x_7 - x_3x_8 + x_4^2 + x_4x_5 + x_4x_7 - x_5^2 - 2x_5x_6 - x_5x_7 + \\
&\quad +x_5x_8 - x_6x_7 + x_6x_8 - x_7^2 + x_7x_8=0, \\
&-2x_1x_3 - x_1x_4 + x_1x_5 - x_1x_7 + 2x_1x_8 + x_2^2 + x_2x_3 - x_2x_4 - x_2x_7 + x_3x_4 + x_3x_5 + \\
&\quad +2x_3x_6 - 2x_3x_7 + 2x_3x_8 - x_4^2 + 2x_4x_5 + 2x_4x_7 - x_4x_8 - x_5^2 + 2x_5x_7 - x_5x_8 + \\
&\quad +2x_6x_7 - 2x_6x_8 + 2x_7x_8 - 2x_8^2=0, \\
&-x_1x_2 + 2x_1x_4 - x_1x_6 + x_1x_7 + x_1x_8 - x_2^2 + 2x_2x_4 + x_2x_5 - x_2x_6 + 2x_2x_7 + \\
&\quad +2x_2x_8 - x_4x_6 - x_4x_7 - x_4x_8 + x_5x_6 + x_5x_7 + x_5x_8=0, \\
&x_1x_3 + 2x_1x_4 - x_1x_5 - x_1x_6 + x_1x_7 + x_1x_8 - x_2^2 - x_2x_3 - x_2x_4 + x_2x_5 + x_2x_6 + \\
&\quad +x_2x_7 - 2x_2x_8 - x_3^2 + 2x_3x_5 + x_3x_6 + x_3x_7 - x_3x_8 + x_4x_5 - x_4x_6 - x_4x_7 - x_4x_8 + \\
&\quad - x_5x_6 + x_5x_7 + 2x_5x_8 - x_6x_7 + x_6x_8 - x_7^2 + x_7x_8=0, \\
&-x_1^2 + x_1x_2 + 2x_1x_3 - x_1x_4 + x_1x_6 - x_1x_7 - x_2x_3 - 2x_2x_4 - 2x_2x_5 - x_2x_7 - x_2x_8 + \\
&\quad - x_3^2 - x_3x_5 + x_3x_7 - x_3x_8 + x_4^2 + x_4x_5 + 2x_4x_7 + x_4x_8 + x_5x_6 + x_5x_7 - x_5x_8 + \\
&\quad +2x_6x_8 + 2x_7^2 + 2x_7x_8 - 2x_8^2=0, \\
&x_1^2 + 2x_1x_2 - x_1x_3 - x_1x_4 + x_1x_6 - x_1x_8 - x_2^2 + 2x_2x_3 - 2x_2x_5 + x_2x_7 + 3x_3^2 - x_3x_4 + \\
&\quad - 2x_3x_6 - x_3x_7 - x_4^2 + 3x_4x_6 + 2x_5^2 + x_5x_6 + x_5x_7 - 2x_6^2 - x_6x_7 + x_6x_8 - x_7^2 + \\
&\quad - 2x_7x_8 + 2x_8^2=0, \\
&2x_1^2 - 2x_1x_2 + x_1x_4 + 3x_1x_5 - 2x_1x_6 - 2x_1x_7 - 2x_1x_8 + x_2^2 - x_2x_3 - 3x_2x_5 - x_2x_7 + \\
&\quad - 3x_3x_4 + x_3x_6 + x_3x_8 + x_4^2 + 3x_4x_5 - 2x_4x_6 + x_4x_7 + x_4x_8 + 2x_5^2 - 4x_5x_6 - 2x_5x_8 + \\
&\quad +2x_6x_7 + x_7^2 - 2x_7x_8 + x_8^2=0.
\end{align*}
\allowdisplaybreaks[0]

We know that this curve doesn't have rational points. The map $\pi_{\textnormal{ns}}\colon \Xns(13)\to \Xns^+(13)$, using the previous model for $\Xns(13)$ and the model (\ref{BurcuEq}) for $\Xns^+(13)$, is
\[
\begin{sistema}
X=-3x_1+2x_2 \\
Y=-3x_1+x_2+2x_4-2x_5 \\
Z=x_1+x_2+x_4-x_5.
\end{sistema}
\]

\bibliographystyle{amsalpha}
\bibliography{Cartan13_arXiv_1}{}

\def\cprime{$'$}
\providecommand{\bysame}{\leavevmode\hbox to3em{\hrulefill}\thinspace}
\providecommand{\MR}{\relax\ifhmode\unskip\space\fi MR }
% \MRhref is called by the amsart/book/proc definition of \MR.
\providecommand{\MRhref}[2]{%
  \href{http://www.ams.org/mathscinet-getitem?mr=#1}{#2}
}
\providecommand{\href}[2]{#2}
\begin{thebibliography}{BGJGP05}

\bibitem[Bar10]{BaranClass}
B.~Baran, \emph{Normalizers of non-split {C}artan subgroups, modular curves,
  and the class number one problem}, J. Number Theory \textbf{130} (2010),
  no.~12, 2753--2772. \MR{2684496 (2011i:11083)}

\bibitem[Bar14]{Baran13}
\bysame, \emph{An exceptional isomorphism between modular curves of level 13},
  J. Number Theory \textbf{145} (2014), 273--300.

\bibitem[BCFS]{MAGMA}
W.~Bosma, J.~J. Cannon, C.~Fieker, and A.~Steel, \emph{Handbook of magma
  functions}, \newline\url{http://magma.maths.usyd.edu.au/magma/handbook/}.

\bibitem[BGJGP05]{BGGP05}
M.~H. Baker, E.~Gonz{\'a}lez-Jim{\'e}nez, J.~Gonz{\'a}lez, and B.~Poonen,
  \emph{Finiteness results for modular curves of genus at least 2}, Amer. J.
  Math. \textbf{127} (2005), no.~6, 1325--1387. \MR{2183527 (2006i:11065)}

\bibitem[BPS16]{BPSgenus3}
N.~Bruin, B.~Poonen, and M.~Stoll, \emph{Generalized explicit descent and its
  application to curves of genus 3}, Forum Math. Sigma \textbf{4} (2016), e6,
  80. \MR{3482281}

\bibitem[Che98]{Chen}
I.~Chen, \emph{The {J}acobians of non-split {C}artan modular curves}, Proc.
  London Math. Soc. (3) \textbf{77} (1998), no.~1, 1--38. \MR{1625491
  (99m:11068)}

\bibitem[DFGS14]{DFGS}
V.~Dose, J.~Fern{\'a}ndez, J.~Gonz{\'a}lez, and R.~Schoof, \emph{The
  automorphism group of the non-split {C}artan modular curve of level 11}, J.
  Algebra \textbf{417} (2014), 95--102. \MR{3244639}

\bibitem[Dos16]{DoseCartan}
V.~Dose, \emph{On the automorphisms of the nonsplit {C}artan modular curves of
  prime level}, Nagoya Math. J. \textbf{224} (2016), no.~1, 74--92.
  \MR{3572750}

\bibitem[DS05]{DS}
F.~Diamond and J.~Shurman, \emph{A first course in modular forms}, Graduate
  Texts in Mathematics, vol. 228, Springer-Verlag, New York, 2005. \MR{2112196
  (2006f:11045)}

\bibitem[dSE00]{Edix}
B.~de~Smit and B.~Edixhoven, \emph{Sur un r\'esultat d'{I}min {C}hen}, Math.
  Res. Lett. \textbf{7} (2000), no.~2-3, 147--153. \MR{1764312 (2001j:11043)}

\bibitem[FNS17]{Fermat23n}
N.~Freitas, B.~Naskręcki, and M.~Stoll, \emph{The generalized fermat equation
  with exponents {$2$}, {$3$}, {$n$}}, preprint (2017), arXiv:1703.05058.

\bibitem[GH78]{GH}
P.~Griffiths and J.~Harris, \emph{Principles of algebraic geometry},
  Wiley-Interscience [John Wiley \& Sons], New York, 1978, Pure and Applied
  Mathematics. \MR{507725 (80b:14001)}

\bibitem[Has97]{Has97}
Y.~Hasegawa, \emph{Hyperelliptic modular curves {$X^*_0(N)$}}, Acta Arith.
  \textbf{81} (1997), no.~4, 369--385. \MR{1472817}

\bibitem[Ken80]{Kenku169}
M.~A. Kenku, \emph{The modular curve {$X_{0}(169)$}\ and rational isogeny}, J.
  London Math. Soc. (2) \textbf{22} (1980), no.~2, 239--244. \MR{588271}

\bibitem[Ken81]{Kenku169errata}
\bysame, \emph{Corrigendum: ``{T}he modular curve {$X_{0}(169)$}\ and rational
  isogeny''\ [{J}. {L}ondon {M}ath. {S}oc. (2) {\bf 22} (1980), no. 2,
  239--244;\ {MR} 81m:10048]}, J. London Math. Soc. (2) \textbf{23} (1981),
  no.~3, 428. \MR{616547}

\bibitem[Mer17]{Mer0}
P.~Mercuri, \emph{Equations and rational points of the modular curves
  {$X_0^+(p)$}}, accepted by The Ramanujan Journal (2017), DOI
  10.1007/s11139-017-9925-2.

\bibitem[MS17]{Merns}
P.~Mercuri and R.~Schoof, \emph{Modular forms invariant under non-split cartan
  subgorups}, preprint (2017).

\bibitem[Ogg74]{OggHyp}
A.~P. Ogg, \emph{Hyperelliptic modular curves}, Bull. Soc. Math. France
  \textbf{102} (1974), 449--462. \MR{0364259 (51 \#514)}

\bibitem[SD73]{Sd73}
B.~Saint-Donat, \emph{On {P}etri's analysis of the linear system of quadrics
  through a canonical curve}, Math. Ann. \textbf{206} (1973), 157--175.
  \MR{0337983}

\bibitem[Ser72]{Ser}
J.-P. Serre, \emph{Propri\'et\'es galoisiennes des points d'ordre fini des
  courbes elliptiques}, Invent. Math. \textbf{15} (1972), no.~4, 259--331.
  \MR{0387283 (52 \#8126)}

\bibitem[Ser97]{SerreMordell}
Jean-Pierre Serre, \emph{Lectures on the {M}ordell-{W}eil theorem}, third ed.,
  Aspects of Mathematics, Friedr. Vieweg \& Sohn, Braunschweig, 1997,
  Translated from the French and edited by Martin Brown from notes by Michel
  Waldschmidt, With a foreword by Brown and Serre. \MR{1757192}

\bibitem[Ste12]{SteinTab}
W.~Stein, \emph{The {M}odular {F}orms {D}atabase},
  \newline\url{http://wstein.org/Tables}.

\bibitem[Sti09]{Stichtenoth}
H.~Stichtenoth, \emph{Algebraic function fields and codes}, second ed.,
  Graduate Texts in Mathematics, vol. 254, Springer-Verlag, Berlin, 2009.
  \MR{2464941}

\bibitem[Zyw15]{ZywinaPossible}
D.~J. Zywina, \emph{On the possible images of the mod {$\ell$} representations
  associated to elliptic curves over {$\mathbb Q$}}, preprint (2015),
  arXiv:1508.07660.

\end{thebibliography}

\end{document}